\documentclass[11pt]{amsart}

 \usepackage{amsmath,amsthm,amsfonts,amssymb,verbatim}
 \usepackage{url}
 \usepackage{graphicx}
 \usepackage[all]{xy}
 \usepackage{wrapfig}
 \usepackage{pinlabel}
 \usepackage{subfigure}
 
\def\co{\colon\thinspace}

\newcommand{\Br}{\mathbf{\Sigma}}

\newcommand{\ab}{\operatorname{ab}}

\newcommand{\sQ}{\mathcal{Q}}

\newcommand{\bN}{\mathbb{N}}
\newcommand{\bZ}{\mathbb{Z}}
\newcommand{\bQ}{\mathbb{Q}}

\newcommand{\si}{\sigma}

\newcommand{\HFhat}{\widehat{\operatorname{HF}}}

\newcommand{\Kh}{\operatorname{Kh}}
\newcommand{\HFK}{\widehat{\operatorname{HFK}}}
\newcommand{\rk}{\operatorname{rk}}

\newcommand{\Char}{\textup{Char}}

\newcommand{\spinc}{\operatorname{Spin}^c}

\newcommand{\frks}{{\mathfrak s}}
\newcommand{\frkt}{{\mathfrak t}}

\newtheorem{theorem}{Theorem}

\newtheorem{definition}[theorem]{Definition}

\newtheorem{corollary}[theorem]{Corollary}
\newtheorem{proposition}[theorem]{Proposition}

\newcommand{\positive}
	{\raisebox{-2pt}{\includegraphics[scale=0.085]{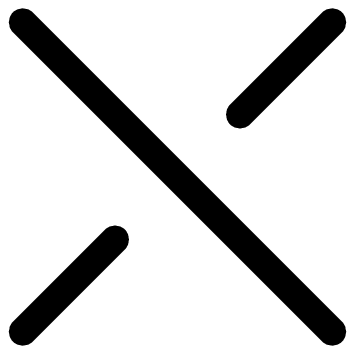}}}
\newcommand{\zero}
	{\raisebox{-2pt}
	{\includegraphics[scale=0.085]{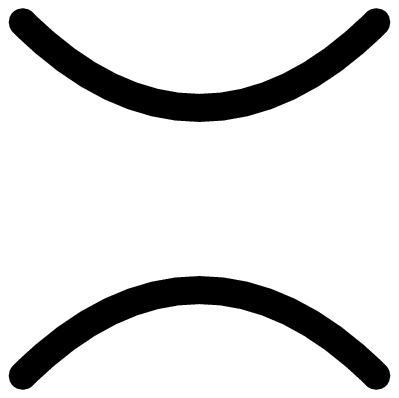}}}
\newcommand{\one}
	{\raisebox{-2pt}
	{\includegraphics[scale=0.085]{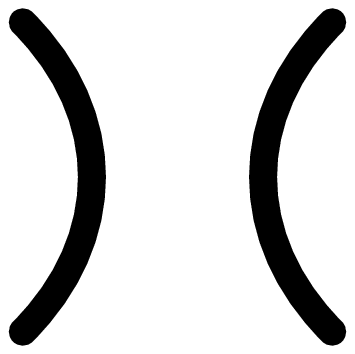}}}

\title[Turaev Torsion, definite 4-manifolds and quasi-alternating knots]{Turaev Torsion, definite 4-manifolds, and quasi-alternating knots}
\date{June 27, 2011}

\author[Greene]{Joshua Evan Greene}
\thanks{Josh Greene was partially supported by an NSF postdoctoral fellowship}
\address{Department of Mathematics, Columbia University, 2990 Broadway, New York, NY 10027.}
\email{josh@math.columbia.edu }

\author[Watson]{Liam Watson}
\thanks{Liam Watson was partially supported by an NSERC postdoctoral fellowship}
\address{Department of Mathematics, UCLA, 520 Portola Plaza, Los Angeles, CA 90095.}
\email{lwatson@math.ucla.edu}

\begin{document}

\begin{abstract}
We construct an infinite family of hyperbolic, homologically thin knots that are not quasi-alternating. To establish the latter, we argue that the branched double-cover of each knot in the family does not  bound a negative definite 4-manifold with trivial first homology and bounded second betti number. This fact depends in turn on information from the correction terms in Heegaard Floer homology, which we establish by way of a relationship to, and calculation of, the Turaev torsion. 
\end{abstract}

\maketitle

\section{Introduction.}\label{s: intro}

Quasi-alternating (QA) links provide a natural extension of the class of alternating links.  They first arose in the context of Heegaard Floer homology of branched double-covers \cite{OSz2005-branch}.

\begin{definition}\label{def:qa}The set of QA links $\sQ$ is the smallest set of links containing the trivial knot that is closed under the following relation: if $L$ admits a projection with distinguished crossing $L(\positive)$ so that $\det(L(\positive))=\det(L(\zero))+\det(L(\one))$ where $L(\zero),L(\one)\in\sQ$, then $L=L(\positive)\in\sQ$ as well. \end{definition}

QA links are {\em thin} by a result of Manolescu and Ozsv\'ath \cite{MO2008}: each of their reduced ordinary Khovanov \cite{Khovanov2000}, odd-Khovanov \cite{ORS2007}, and $\bZ/2\bZ$ knot Floer \cite{OSz2004-knot,Rasmussen2003} homology groups is torsion-free and supported on a single diagonal with respect to the theory's bigrading.  However, it was shown in \cite{Greene2010} that the converse does not hold: the thin knot $11^n_{50}$ is non-QA.

The purpose of the present paper is to exhibit further examples of this kind in the following strong sense.

\begin{theorem}\label{thm:main}
There exists an infinite family of thin, hyperbolic, non-QA knots with identical homological invariants.
\end{theorem}

The examples come from a construction of Kanenobu  \cite{Kanenobu1986} depicted in Figure \ref{fig:Kanenobu} (see also \cite{Watson2006,Watson2007}).  The main effort is to show that these knots are non-QA, which we establish by studying the Turaev torsion of the branched double-cover.

We now provide a brief overview of the argument.  For a QA link $L$, Ozsv\'ath and Szab\'o observed that $\Br(L)$, the branched double-cover of $L$, bounds a particular type of 4-manifold $W$ \cite{OSz2005-branch}.  Specifically, $H_1(W;\bZ)=0$, the intersection pairing $Q_W$ on $W$ is negative-definite, and both its rank and discriminant are bounded by the link determinant $\det(L)$.  By a basic result of Eisenstein and Hermite, it follows that the lattice $\Lambda := (H_2(W),Q_W)$ belongs to one of finitely many isomorphism types \cite{MilnorHusemoller1973}.  Associated with a lattice is a certain numerical invariant $m(\Lambda) \in \bQ$, so the finiteness result implies an absolute lower bound on $m(\Lambda)$ in terms of $\det(L)$.  Now, Ozsv\'ath and Szab\'o defined a collection of numerical invariants for a 3-manifold called its {\em correction terms}, and showed in the setting at hand that $m(\Lambda)$ provides a lower bound on the correction terms of $\Br(L)$ \cite{OSz2003-grading}.  In summary, there exists an absolute lower bound, in terms of $\det(L)$, on the correction terms of $\Br(L)$ for a QA link $L$.

Thus, our strategy is to study an infinite family of thin, hyperbolic knots $K_n$ with identical homological invariants, and argue that the smallest correction term of $\Br(K_n)$ tends to $-\infty$ with $n$.  Since these knots have the same determinant, it follows that taking all $K_n$ with $n$ sufficiently large provides the desired family to establish Theorem \ref{thm:main}.

To obtain the result about the correction terms, it suffices, by results of Mullins \cite{Mullins1993} and Rustamov \cite{Rustamov1}, to show that the smallest coefficient in the Turaev torsion $\tau(\Br(K_n))$ tends to $-\infty$ with $n$.  In order to show this, we present the space $\Br(K_n)$ by a relatively simple Heegaard diagram and establish the behaviour of the torsion invariant directly from its definition. 

The remainder of the paper is organized as follows. In Section \ref{s: setup} we formalize the obstruction sketched above, recalling in particular the necessary background about correction terms.  In Section \ref{s: kanenobu} we define the family of knots $K_n$ and collect their basic properties: we establish that these knots are thin, hyperbolic, and possess identical homological invariants.  Finally, in Section \ref{s: turaev} we study the topology of the spaces $\Br(K_n)$ and calculate their Turaev torsion, completing the proof of Theorem \ref{thm:main}.

\subsection*{Acknowledgement.}  Thanks to Matt Hedden for helpful conversations, and especially his input to Theorem \ref{thm:hfk}.

\section{An obstruction.}\label{s: setup}

Our obstruction to QA-ness reads as follows.

\begin{proposition}\label{prp:qa}
For all $D \in \bN$, there exists a constant $C = C(D) \in \bZ$ such that if $L$ is a QA link with $\det(L) = D$, then
\[ C \leq d(\Br(L),\frkt) \quad \forall \, \frkt \in \spinc(\Br(L)).\]
\end{proposition}

Here $d(Y,\frkt) \in \bQ$ denotes the {\em correction term} or {\em d-invariant} for a 3-manifold $Y$ equipped with a torsion spin$^c$ structure $\frkt$.  It was defined by Ozsv\'ath and Szab\'o in Heegaard Floer homology by analogy to the Fr\o yshov $h$-invariant in Seiberg-Witten theory \cite{OSz2003-grading}.

The proof of Proposition \ref{prp:qa} rests on three facts.  The first of these is implicit in the work of Ozsv\'ath and Szab\'o \cite{OSz2005-branch}.

\begin{theorem}[Ozsv\'ath-Szab\'o {\cite[Proposition 3.3 and Proof of Lemma 3.6]{OSz2005-branch}}]\label{thm:os1}

If $L$ is QA, then $\Br(L)$ is an L-space that bounds a negative definite $4$-manifold $W$ with $H_1(W)=0$ and $b_2(W)<\det(L)$.

\end{theorem}

\noindent Recall that an {\em L-space} is a rational homology sphere $Y$ with the property that $|H_1(Y;\bZ)|= \rk\HFhat(Y)$. Here and throughout we work with $\bZ / 2 \bZ$ coefficients for the Heegaard Floer homology group $\HFhat$.

\begin{proof}[Sketch of the Proof of Theorem \ref{thm:os1}.]
We proceed by induction on $\det(L)$.  When $\det(L) = 1$, $L$ is the unknot, $\Br(L) = S^3$ is an L-space, and we take $W = B^4$.  Now, given a QA link with $\det(L) > 1$, choose a QA crossing and let $L_0, L_1 \in \sQ$ denote its two resolutions.  By induction, both $\Br(L_0)$ and $\Br(L_1)$ are L-spaces, so the same follows for $\Br(L)$ from the skein sequence in $\HFhat$ relating these three spaces. Furthermore, from the skein sequence we obtain a negative definite 2-handle cobordism $X$ from $\Br(L_i)$ to $\Br(L)$ and another from $\Br(L)$ to $\Br(L_j)$, where $\{i,j\} = \{0,1\}$.  By induction, $\Br(L_i)$ bounds a negative-definite 4-manifold $W_i$ with $H_1(W_i) = 0$ and $b_2(W_i) < \det(L_i)$.  It follows that $W = W_i \cup X$ satisfies the conclusions of the Theorem for $L$, which completes the induction step. 
\end{proof}

The second fact is a classical result that implies that the intersection pairing on $W$ in Theorem \ref{thm:os1} belongs to one of finitely many isomorphism types.

\begin{theorem}[Eisenstein-Hermite {\cite[Lemma 1.6]{MilnorHusemoller1973}}]\label{thm:lattice}

There exist finitely many isomorphism types of definite, integral lattices with bounded rank and discriminant.  
\qed

\end{theorem}

The third and final fact is also due to Ozsv\'ath and Szab\'o.

\begin{theorem}[Ozsv\'ath-Szab\'o {\cite[Theorem 9.6]{OSz2003-grading}}]\label{thm:os2}
Suppose that a rational homology sphere $Y$ bounds a negative definite 4-manifold $W$ with $H_1(W) = 0$.  Then every $\frkt \in \spinc(Y)$ extends to some $\frks \in \spinc(W)$, and we have
\[ c_1(\frks)^2 + b_2(W) \leq 4 d(Y,\frkt).\rlap{\hspace{1.86in} \qed}\] 
\end{theorem}

\begin{proof}[Proof of Proposition \ref{prp:qa}]

Fix a value $D \in \bN$ and select a QA link $L$ with $\det(L)=D$. Set $Y = \Br(L)$, choose a 4-manifold $W$ as in Theorem \ref{thm:os1}, and let $\Lambda$ denote the lattice $(H_2(W),Q_W)$.  The values $c_1(\frks)$, $\frks \in \spinc(W)$, constitute the set of characteristic covectors
\[ \Char(\Lambda) := \{ \chi \in \text{Hom}(\Lambda,\bZ) \; | \; \langle \chi, v \rangle \equiv \langle v,v \rangle \; (\text{mod} \, 2), \, \forall v \in \Lambda \}. \]
Furthermore, the different subsets $\Char(\Lambda,\frkt) := \{ c_1(\frks) \; | \; \frks | Y = \frkt \} \subset \Char(\Lambda)$, $\frkt \in \spinc(Y)$, constitute the different equivalence classes in $\Char(\Lambda)$ $(\text{mod} \, 2 \Lambda)$, of which there are $ \text{disc}(\Lambda) =D$.  Let $m(\Lambda)$ denote the minimum value of 
\[ \max \{ (\chi^2 + \rk(\Lambda))/4 \; | \; \chi \in \Char(\Lambda,\frkt) \} \]
over the $D$ equivalence classes $ \Char(\Lambda,\frkt)$. It follows from Theorem \ref{thm:os2} that $m(\Lambda)$ provides a lower bound on $d(Y,\frkt)$, $\forall \, \frkt \in \spinc(Y)$.  Set
\[ C(D) = \inf \{ m(\Lambda') \; | \; \Lambda' \textup{ integral, definite}, \, \rk(\Lambda') < \text{disc}(\Lambda') = D \}.\]
By Theorem \ref{thm:lattice}, $C(D)$ is finite, and we obtain $C(D) \leq d(Y,\frkt), \forall \, \frkt \, \in \spinc(Y)$.  Since $L$ was arbitrary, it follows that the value $C(D)$ provides the desired constant.
\end{proof}

In the next section, we introduce the knots $K_n$ to which we will apply Proposition \ref{prp:qa}.


\section{Kanenobu's knots.}\label{s: kanenobu}


We begin by collecting results about the family of knots $K_{p,q}$ in Figure \ref{fig:Kanenobu}, adhering to the convention that $p=1$ (or $q=1$) denotes the half-twist $\positive$ (this convention follows \cite{Watson2006,Watson2007}). This family of knots slightly generalizes a construction of Kanenobu \cite{Kanenobu1986}, who restricted attention to even values for $p,q$. Kanenobu observed that the knots in this family are ribbon (the band sum of two discs), hence slice, and this observation remains true in the present setting. We focus on the homological invariants of these knots, denoting the Khovanov, odd-Khovanov, and knot Floer homology groups by $\Kh$, $\Kh^{\rm odd}$, and $\HFK$, respectively.


\begin{figure}[ht!]
\begin{center}
\labellist 
\small
	\pinlabel $\cdots$ at 184 432
	\pinlabel $\cdots$ at 442 432
	\pinlabel $\underbrace{\phantom{aaaaaaaaaa}}_p$ at 149 376
	\pinlabel $\underbrace{\phantom{aaaaaaaaaAa}}_q$ at 390 376
\endlabellist
\raisebox{-5pt}{\includegraphics[scale=0.4]{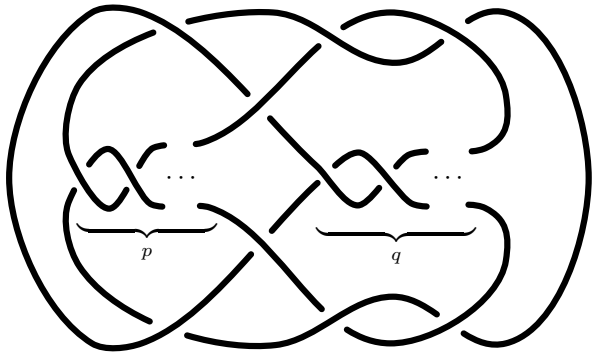}}
\caption{Kanenobu's knot $K_{p,q}$, where $p$ and $q$ denote the number of half-twists.}\label{fig:Kanenobu}
\end{center}\end{figure}

\begin{theorem}\label{thm:even} For all $p,q\in\bZ$, $\Kh(K_{p,q})\cong\Kh(K_{p+1,q-1})$.\qed\end{theorem}

This follows from a computation using the skein exact sequence, together with Lee's spectral sequence, and is the focus of \cite{Watson2007} (in particular, see \cite[Sections 3 and 7.4]{Watson2007}). As an immediate consequence (compare \cite{Watson2006}) we obtain:

\begin{corollary}\label{crl:jones}  There is an equality of Jones polynomials $V_{K_{p,q}}(t)=V_{K_{p+1,q-1}}(t)$ for any $p,q\in\bZ$.\qed\end{corollary}

This corollary is useful in establishing the following: 

\begin{theorem}\label{thm:odd}For any $p,q\in\bZ$, $\Kh^{\rm odd}(K_{p,q})\cong\Kh^{\rm odd}(K_{p+1,q-1})$.\end{theorem}
\begin{proof}
Recall that $\Kh^{\rm odd}$ satisfies the same skein exact sequence as $\Kh$ \cite[Proposition 1.5]{ORS2007}. Thus, following the proof of Theorem \ref{thm:even}, we immediately have that  $\Kh^{\rm odd}(K_{p,q})\cong\Kh^{\rm odd}(K_{p+1,q-1})$ for all homological gradings not equal to zero  \cite[Lemma 3]{Watson2007}. Moreover, the torsion is isomorphic without this grading restriction, so we work over $\bQ$ for the remainder of the argument (see \cite[pp.1398--1399, Proof of Lemma 3]{Watson2007}). However, as there is no analogue to Lee's spectral sequence in this theory \cite[Section 5]{ORS2007}, the final step in the proof must be altered as follows. 

Since $\Kh^{\rm odd}(K_{p,q})\cong\overline{\Kh}(K_{p,q})[0,+1]\oplus \overline{\Kh}(K_{p,q})[0,-1]$  \cite[Proposition 1.7]{ORS2007}, where $[i,j]$ denotes a grading shift in the homological grading $i$ and quantum grading $j$, we proceed without loss of generality by considering $\overline{\Kh}(K_{p,q})$. Now  $V_{K_{p,q}}(t)=V_{K_{p+1,q-1}}(t)$ by Corollary \ref{crl:jones}, so we have the equality of graded Euler characteristics
\[\sum_{i,j}(-1)^i\rk\overline{\Kh}_j^i(K_{p,q})t^j = \sum_{i,j}(-1)^i\rk\overline{\Kh}_j^i(K_{p+1,q-1})t^j. \]
In combination with the observation $\overline{\Kh}_j^i(K_{p,q})\cong\overline{\Kh}_j^i(K_{p+1,q-1})$ for $(i,j)\ne(0,0)$ (again, see \cite[Proof of Lemma 3]{Watson2007}), we conclude that $\rk \overline{\Kh}_0^0(K_{p,q})= \rk \overline{\Kh}_0^0(K_{p+1,q-1})$, completing the argument. \end{proof}

The behaviour of the Alexander polynomial is slightly different in this setting since the parity of $p$ comes to bear.  There is an oriented skein triple involving $K_{p,q}$, the 2-component unlink, and either of $K_{p+2,q}$ or $K_{p,q+2}$.  From the skein relation we therefore obtain $\Delta_{K_{p,q}}(t) = \Delta_{K_{p+2,q}}(t) = \Delta_{K_{p,q+2}}(t)$. In a similar spirit, the skein exact sequence in knot Floer homology establishes the following:
 
\begin{theorem}\label{thm:hfk}
For all $p,q\in\bZ$, $\HFK(K_{p,q}) \cong \HFK(K_{p+2,q}) \cong \HFK(K_{p,q+2})$.
\end{theorem}
\begin{proof} Since the knots $K_{p,q}$ are ribbon, this result is a special case of an observation due to Matthew Hedden. Noting that the Ozsv\'ath-Szab\'o concordance invariant $\tau(K_{p,q})$ vanishes, we have the isomorphism $\operatorname{HFK}^-(K_{p,q}) \cong \operatorname{HFK}^-(K_{p+2,q})$ as $\bZ[U]$-modules by an application of the skein exact sequence. It follows that $\HFK(K_{p,q})\cong\HFK(K_{p+2,q})$, and similarly that $\HFK(K_{p,q})\cong\HFK(K_{p,q+2})$, as claimed.
\end{proof}

We now restrict attention to the infinite family of knots \[K_n=K_{-10n,10n+3},\] $n\ge0$. We remark that $K_0$ is the knot $11^n_{50}$ -- the central example of \cite{Greene2010}.  The knots $K_n$ are distinguished by the Turaev torsion of $M_n=\Br(K_n)$, as we show in the next section.  We conclude by summarizing the relevant properties of $K_n$.

\begin{proposition}\label{prp:K_n-properties} The knots $K_{n}$ are ribbon, hyperbolic, and have identical Khovanov, odd-Khovanov, and knot Floer invariants. In particular, the knot $K_n$ is thin and $M_n$ is an L-space for all $n\ge0$. \end{proposition}

\begin{proof}
We have already noted that the knots $K_{p,q}$ are ribbon.  Theorems \ref{thm:even}, \ref{thm:odd}, and \ref{thm:hfk} collectively establish that, for all $n > 0$, the homological invariants of $K_n$ agree with that of $K_0$, which was observed to be thin in \cite{Greene2010}.  From the spectral sequence relating $\Kh(L)$ and $\HFhat(-\Sigma(L))$, it follows that $M_n$ is an L-space for all $n \geq 0$.

It stands to show that $K_n$ is hyperbolic, which we do by adapting the argument of \cite[Lemma 5]{Kanenobu1986}.  Figure \ref{fig:Kanenobu} exhibits a 3-bridge diagram for $K_n$.  This knot has determinant 25 and (checking $\HFK$) Seifert genus 2.  The only 2-bridge knot with these invariants is $8_8$, and $\Delta_{8_8}(t) \ne \Delta_{K_n}(t)$, so $K_n$ is 3-bridge.  Now a result of Riley implies that $K_n$ is either composite, a torus knot, or hyperbolic \cite{Riley1979}. If it were composite, then it would be a connected sum of a pair of 2-bridge knots, and the branched double-cover would be a non-trivial connected sum of lens spaces. However, this possibility is ruled out by the cyclic first homology group $\bZ/25\bZ$, which we calculate in the next section. It cannot be a torus knot because of its determinant and genus.  Hence $K_n$ is hyperbolic, as claimed.\end{proof}

\section{Turaev torsion.}\label{s: turaev}

\subsection{From $d$ to $\tau$.} We begin by relating the $d$-invariant of an L-space $Y$ to a pair of well-known invariants, the Casson-Walker invariant $\lambda(Y)$ and the Turaev torsion $\tau(Y,\frkt)$.

\begin{theorem}[Rustamov {\cite[Theorem 3.4]{Rustamov1}}]\label{thm:rustamov}
For an L-space $Y$ and $\frkt \in \spinc(Y)$, we have
\[ d(Y,\frkt) = 2 \tau(Y,\frkt) - \lambda(Y).\rlap{\hspace{1.89in} \qed}\] 
\end{theorem}

Here we normalize so that $\lambda(P) = -2$, where $P$ denotes the Poincar\'e homology sphere, oriented as the boundary of the negative definite $E_8$ plumbing.

For the case of a branched double-cover, we calculate the Casson-Walker invariant by the following formula.

\begin{theorem}[Mullins {\cite[Theorem 5.1]{Mullins1993}}]\label{thm:mullins}
For a link $L$ with $\det(L) \ne 0$, we have
\[ \lambda(\Br(L)) = -\frac{V'_L(-1)}{6V_L(-1)} + \frac{\sigma(L)}{4}. \rlap{\hspace{1.70in} \qed}\] 
\end{theorem}

Here $V_L(t)$ denotes the Jones polynomial and $\sigma(L)$ the signature of the link $L$.  These invariants both depend for their definition on a choice of orientation of $L$ when $L$ has multiple components, although $\lambda(\Br(L))$ does not.  Note that for $L$ the positive $(3,5)$-torus knot, we obtain $V'_L(-1) = 0$, $\sigma(L) = -8$, and $\Br(L) \cong P$, which is consistent with $\lambda(P) = -2$.

We now apply these results to the space $M_n = \Br(K_n)$.  By Proposition \ref{prp:K_n-properties}, the polynomial $V_{K_n}$ is independent of $n$, and since $K_n$ is ribbon, the signature $\si(K_n)$ vanishes.  By Theorem \ref{thm:mullins}, it follows that $\lambda(M_n)$ is a constant $\lambda \in \bQ$ independent of $n$. Since $M_n$ is an L-space, Theorem \ref{thm:rustamov} applies to show that
\begin{equation}\label{e: d and tau}
d(M_n,\frkt) = 2\tau(M_n,\frkt) - \lambda, \quad \forall \, n \geq 0,\forall \, \frkt \in \spinc(M_n).
\end{equation}

Each $K_n$ has determinant $D = 25$, so Theorem \ref{thm:main} will follow on application of Proposition \ref{prp:qa} to the knots $K_n$ once we establish the following result.
\begin{proposition}\label{prp:turaev-growth}

\begin{equation}\label{e: tau unbounded}
\lim_{n \to \infty} \min \{ \tau(M_n,\frkt) \; | \; \frkt \in \spinc(M_n) \} = -\infty.
\end{equation}

\end{proposition}


The remainder of the paper is devoted to the proof of Proposition \ref{prp:turaev-growth}, constituting the final step in the proof of Theorem \ref{thm:main}. This is accomplished by calculating the Turaev torsion of $M_n$.

Our treatment of the torsion follows Turaev's book \cite{Turaev2002}; we calculate it for a rational homology sphere $M$ through the following steps.  First, present $M$ by a Heegaard diagram.  From the diagram, write down the induced presentation of $\pi_1(M)$ and $H=H_1(M;\bZ)$.  From this presentation, write down the Fox matrix and its abelianization $A$.  Now the determinants of various minors of $A$, which lie in the group ring $\bZ[H]$, determine the torsion of $M$.  More precisely, in the case that $H$ is cyclic of prime power order, it suffices to calculate a single minor $\Delta$.  Then an explicit automorphism of $\bQ[H]$ and identification $\spinc(M) \overset{\sim}{\to} H$ transforms $\Delta$ into the torsion invariant.

\subsection{Presentations for $M_n$.}

Working from the knot projection of $K_n$ displayed in Figure \ref{fig:Kanenobu}, checkerboard colour its regions white and black so that the unbounded region gets coloured white, and construct the corresponding white graph.  Following \cite[Section 3.1]{Greene2008}, the white graph of the diagram for a link $L$ leads naturally to a Heegaard diagram for the space $\Br(L)$.  The resulting presentation of $\pi_1(\Br(L))$ is then given in concise terms from the combinatorics of the white graph.   In the case at hand, we obtain a presentation of the fundamental group with one generator $a_i$ and relator $b_i$ for each vertex $v_i$ of the white graph in a bounded region:
\[\pi_1(M_n)=\langle a_1,a_2,a_3,a_4 \; | \; b_1,b_2,b_3,b_4 \rangle. \]
To obtain the relator $b_i$, traverse a small counterclockwise loop around $v_i$.  For each edge $e$ between $v_i$ and another vertex $v_j$, record the word $(a_j^{-1} a_i)^{\mu(e)}$, where $\mu(e) = \pm 1$ denotes the sign of the crossing corresponding to $e$ (see Figure \ref{fig:conventions}). 
\begin{figure}
[ht!]
\begin{center}
\labellist 
\pinlabel $+$ at 160 380
\pinlabel $-$ at 610 380
\endlabellist
\includegraphics[scale=0.27]{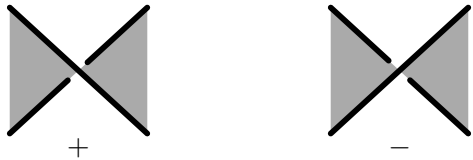}
\caption{Sign conventions at a crossing given a colouring of a knot diagram.}\label{fig:conventions}
\end{center}
\end{figure}
The product of these terms, from left to right, gives the relator $b_i$. With $p = -10n$, $q=10n+3$, we obtain the graph and relators displayed in Figure \ref{fig:relations}.

\begin{figure}
\[
\begin{array}{cc}
\labellist \small
\pinlabel {\tiny $10n$} at 119 407
\pinlabel {\tiny $10n+3$} at 489 407
	\pinlabel $\cdots$ at 119 387
	\pinlabel $\cdots$ at 497 387
	\pinlabel $+$ at 120 577
	\pinlabel $+$ at 432 610
	\pinlabel $+$ at 488 577
	\pinlabel $-$ at 307 219
	\pinlabel $-$ at 202 387
	\pinlabel $-$ at 35 387
	\pinlabel $+$ at 307 570
	\pinlabel $+$ at 407 387
	\pinlabel $+$ at 577 387	
	\pinlabel $-$ at 120 200
	\pinlabel $-$ at 432 190
	\pinlabel $-$ at 488 200
	\pinlabel $v_1$ at  178 518
	\pinlabel $v_2$ at 178 264
	\pinlabel $v_3$ at 437 264
	\pinlabel $v_4$ at 437 516
\endlabellist
\raisebox{-41pt}{\includegraphics[scale=0.27]{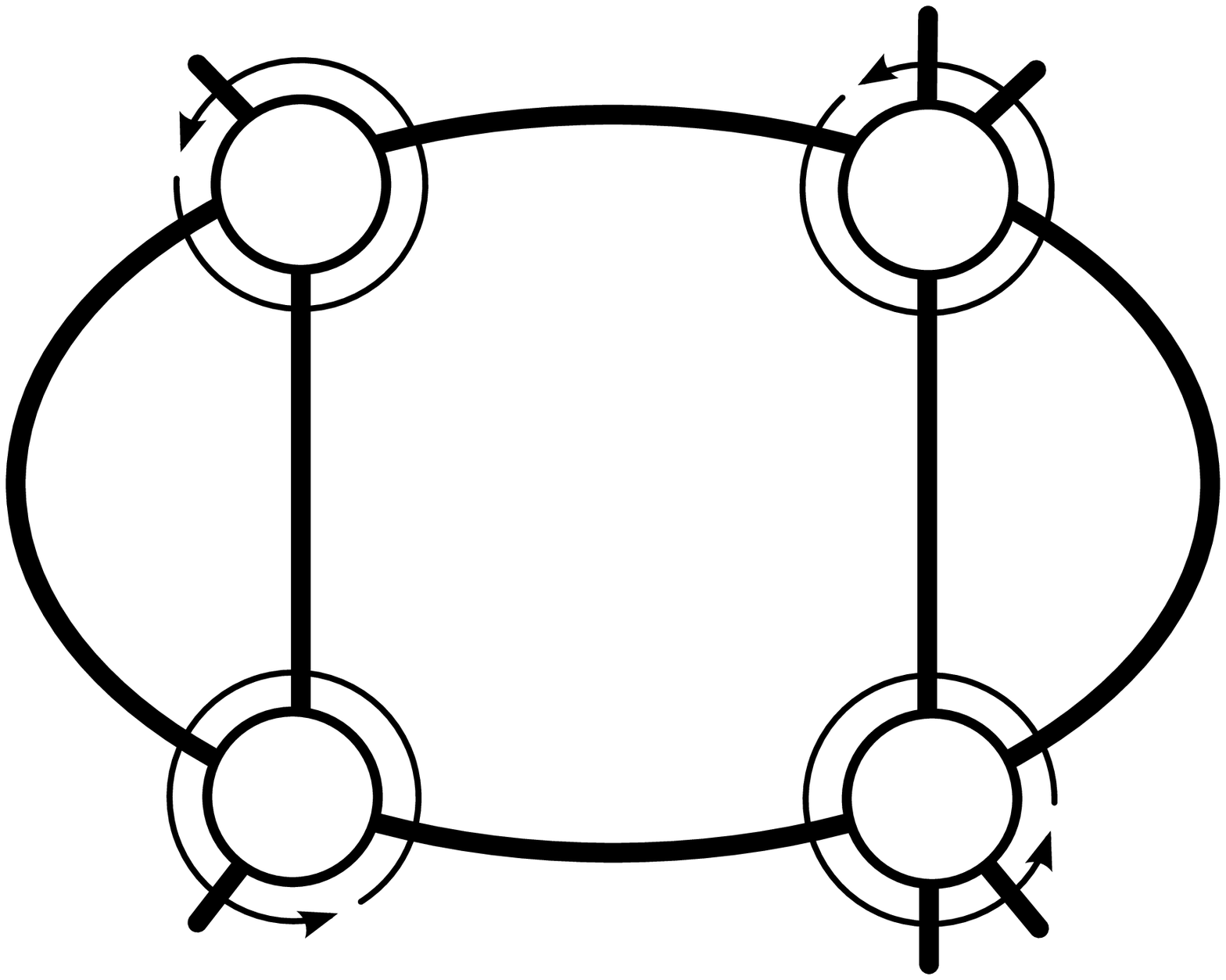}}
\qquad\qquad&
\begin{array}{l}
b_1 = (a_1^{-1}a_2)^{10n}a_4^{-1}a_1^2 \\[4pt]
b_2 =a_2^{-1} a_3(a_2^{-1}a_1)^{10n}a_2^{-1} \\[4pt]
b_3 = (a_4^{-1}a_3)^{10n+3}a^{-1}_3a_2a_3^{-2} \\[4pt]
b_4 = a_1^{-1}a_4(a_3^{-1}a_4)^{10n+3}a_4^2 \\[4pt]
\end{array}\end{array}\]\caption{The reduced white graph for the knot $K_n$ gives a recipe for the relations of the fundamental group $\pi_1(M_n)$ corresponding to a genus 4 Heegaard splitting of $M_n$.}\label{fig:relations}\end{figure}

Under the abelianization map $\ab\co\pi_1(M_n) \to H_1(M_n;\bZ)$, we obtain the presentation matrix 
\[\left( \begin{matrix}
-10n+2 & 10n & 0 & -1 \\
10n & -10n-2 & 1 & 0 \\
0 & 1 & 10n & -10n-3\\
-1 & 0  & -10n-3& 10n+6
\end{matrix} \right)\]
for $H_1(M_n;\bZ)$.  Calculating its cokernel, we find that $H_1(M_n;\bZ)$ is cyclic of order $25$, generated by any of the elements $\ab(a_i)$.  In terms of the fixed choice of generator $t := \ab(a_4)$, we calculate $\ab(a_1) = t^{13}, \ab(a_2) = t^3, \ab(a_3)=t^6$.  Thus we obtain a natural isomorphism $H_1(M_n;\bZ) \cong H := \langle t \; | \; t^{25} \rangle$.

We focus attention on an important pair of generating sets for $H$.  In the Heegaard splitting $U_\alpha \cup_\Sigma U_\beta$ specified by the white graph, let $g_i \in H$ denote the homology class of an oriented curve supported in the handlebody $U_\alpha$ that meets the disc $D_i$ bounded by $\alpha_i$ once positively and avoids all the other $D_j$, and define $h_i \in H$ similarly with respect to $U_\beta$.  In the case at hand, it is straightforward to locate an oriented curve $\gamma_i \subset \Sigma$ that meets both $\alpha_i$ and $\beta_i$ once positively and is disjoint from the other $\alpha_j, \beta_j$.  Thus, $\ab(a_i) = [\gamma_i] = g_i = h_i \in H$.

\subsection{The Fox matrix and its minors.}  With the notation $\partial_i=\frac{\partial}{\partial a_i}$ for the Fox free derivatives, let $F_n=(\partial_i b_j)$ denote the Fox matrix corresponding to the preceding presentation for $\pi_1(M_n)$, with entries in the group ring $\bZ[\pi_1(M_n)]$.  We calculate its $(4,4)$-principal minor as

\[  F_n^{44} = \small \left( \begin{matrix}
\partial_1 b_1 & 
a_2^{-1} a_3 \partial_1 (a_2^{-1} a_1)^{10n} & 
0 \\

\partial_2 (a_1^{-1} a_2)^{10n} &
\partial_2 b_2 &  
 (a_4^{-1}a_3)^{10n+3}a_3^{-1} \\

0 &  
a_2^{-1} & 
\partial_3 b_3 \\
\end{matrix} \right), \]

where

\[ \partial_1 b_1 = \partial_1(a_1^{-1}a_2)^{10n} +(a_1^{-1}a_2)^{10n}a_4^{-1}(1+a_1), \]
\[ \partial_2 b_2 = -a_2^{-1}+a_2^{-1}a_3\partial_2(a_2^{-1}a_1)^{10n}-a_2^{-1}a_3(a_2^{-1}a_1)^{10n}a_2^{-1}, \]
\[ \partial_3 b_3 =  \partial_3(a_4^{-1}a_3)^{10n+3} - (a_4^{-1}a_3)^{10n+3}a_3^{-1}(1+a_2 a_3^{-1} + a_2 a_3^{-2}), \]

and

\[ \partial_j(a_i^{-1}a_j)^{k}=a_i^{-1} \cdot {1 - (a_j a_i^{-1})^k  \over 1 - a_j a_i^{-1}}, \quad \partial_i(a_i^{-1}a_j)^{k}=-\partial_j(a_i^{-1}a_j)^{k}, \quad i\ne j.\]

Extend $\ab$ to a mapping $\bZ[\pi_1(M_n)] \to \bZ[H]$ and apply it to the entries of $F_n^{44}$ to obtain the abelianized minor

\[ A_n^{44} = \small \left( \begin{matrix}
 
 (-n \sigma+1)t^{12}+t^{24} & n \sigma & 0  \\
 n \sigma t^{12} & -n \sigma- 1 -t^{22} & t^9 \\
 0 & t^{22} & (n \sigma+1)t^{24}-1+t^4-t^6
\end{matrix} \right), \] 

writing

\[ \sigma = 2 (1+t^5 + t^{10} + t^{15} + t^{20}).\]

After a little manipulation, we calculate its determinant as

\begin{equation}\label{e: minor}
\Delta_n^{44} = n \sigma (1 + t + t^3) -1 + t^2 - t^3 - t^8 + t^9 - t^{11} + t^{12} - t^{13} + t^{15} - t^{16} - t^{20}+ t^{21} - t^{23} + t^{24}.
\end{equation}

\subsection{From the minor to the torsion.}

For a $\bZ/2\bZ$-homology sphere $M$, we have an identification 
\[ \text{Spin}^c(M) \overset{c_1}{\longrightarrow} H^2(M;\bZ) \overset{\text{PD}}{\longrightarrow} H_1(M;\bZ)\]
via the first Chern class and Poincar\'e duality.  In this way, we regard the Turaev torsion $\tau(M)$ as an element of the group ring $\bQ[H]$.  We abbreviate $\tau_n := \tau(M_n) \in \bQ[H]$, using our fixed identification $H_1(M_n;\bZ) \overset{\sim}{\to} H \cong \bZ / 25 \bZ$.

Following \cite[p.8, I.3.1]{Turaev2002}, we decompose the group ring via the map
\[ \varphi: = (\varphi_0,\varphi_1,\varphi_2) : \bQ[H] \overset{\sim}{\to} \bQ \oplus \bQ(\zeta_5) \oplus \bQ(\zeta_{25}),\]
where $\zeta_k$ denotes a primitive $k^{th}$ root of unity, and $\varphi_j$ is defined by the condition that it maps $t \in H$ to $\zeta_{5^j}$, $j=0,1,2$.  The normalization condition on the torsion implies that $\varphi_0(\tau_n) = 0$.  Additionally, note that the element $\alpha := {1 \over |H|} \sum_{h \in H} h \in \bQ[H]$ maps to $(1,0,0)$.

Given indices $j,r,s$ such that $\varphi_j(g_r), \varphi_j(h_s) \ne 1$, we have
\[ \varphi_j(\tau_n) = \epsilon_{jrs} \cdot (\varphi_j(g_r) - 1)^{-1} (\varphi_j(h_s) - 1)^{-1} \cdot \varphi_j(\Delta^{rs}_n),\]
where $\epsilon_{jrs} \ne 0$ does not depend on $n$ \cite[p.15, item (4)]{Turaev2002}.
In the case at hand, we take $r=s=4$, so $g_4 = h_4 = t$, and obtain
\[ \varphi_j(\tau_n) = \epsilon_{jrs} \cdot (\zeta_{5^j} - 1)^{-2} \cdot \varphi_j(\Delta^{44}_n), \,  j=1,2.\]

Let $\psi$ denote the automorphism of $\bQ \oplus \bQ(\zeta_5) \oplus \bQ(\zeta_{25})$ that acts as the identity on the factor $\bQ$ and as multiplication by $\epsilon_{jrs} \cdot (\zeta_{5^j} - 1)^{-2}$ on the factor $\bQ(\zeta_{5^j})$, $j=1,2$.  Then $\varphi^{-1} \circ \psi \circ \varphi$ is an automorphism of $\bQ[H]$, independent of $n$, that carries $\Delta^{44}_n - \varphi_0(\Delta^{44}_n) \cdot \alpha$ to $\tau_n$.

\subsection{Conclusion of the argument.}  The proofs of Proposition \ref{prp:turaev-growth} and Theorem \ref{thm:main} follow directly from the foregoing material.

\begin{proof}[Proof of Proposition \ref{prp:turaev-growth}.]

By \eqref{e: minor}, $\Delta^{44}_n - \varphi_0(\Delta^{44}_n) \cdot \alpha \in \bQ[H]$ varies linearly in $n$ and is non-constant.  Applying the automorphism $\varphi^{-1} \circ \psi \circ \varphi$ to it, it follows that the same holds for $\tau_n$ as well.  Due to the normalization $\varphi_0(\tau_n) = 0$, \eqref{e: tau unbounded} follows at once.
\end{proof}

\begin{proof}[Proof of Theorem \ref{thm:main}]

An infinite family is given by the knots $K_n$ for $n \gg 0$.  By Proposition \ref{prp:K_n-properties}, these knots are thin, hyperbolic, and have identical homological invariants.  It remains to argue that they are non-QA.  Combining \eqref{e: d and tau} and \eqref{e: tau unbounded}, we have 
\begin{equation*}
\lim_{n \to \infty} \min \{ d(M_n,\frkt) \; | \; \frkt \in \spinc(M_n) \} = -\infty.
\end{equation*}
Since the knots $K_n$ have fixed determinant $D = 25$, Proposition \ref{prp:qa} implies that $K_n$ is non-QA for $n \gg 0$.
\end{proof}

\subsection{Closing remarks.}

We briefly remark on the use of $10 = 2 \times 5$ in the definition of $K_n$.  The factor of 2 ensures that $\HFK(K_n)\cong\HFK(K_{n+1})$ for all $n\ge0$, while the factor of 5 ensures that $H_1(M_n;\bZ)\cong\bZ/25\bZ$ and yields a linear expression for the minor $\Delta^{44}_n$ in terms of $n$.  Any of the other nine other families of knots $K_{-10n-j,10n+j+3}$, $j=1,\dots,9$, $n \gg 0$, should suffice to establish Theorem \ref{thm:main}, with minor changes.



Since the knot $K_n$ is 3-bridge, the manifold $M_n$ admits a Heegaard decomposition of genus two. This splitting gives an alternative presentation from which the torsion invariant may be calculated.  On the other hand, the chosen genus four Heegaard splitting (and associated presentation for the fundamental group) generalizes in a straightforward manner to the symmetric union of any pair of twist knots (note that $K_{p,q}$ is the symmetric union of figure eight knots). These symmetric unions give a natural extension of the class of Kanenobu knots, and, in particular, the various homological invariants within a given family are identical (this is established in \cite{Watson2007}). Thus, a version of Proposition \ref{prp:K_n-properties} applies to these knots, yielding further infinite families to which the techniques of this paper should apply to produce examples of hyperbolic, thin, non-QA knots with identical homological invariants.

Finally, we recall and promote \cite[Conjecture 3.1]{Greene2010}, which asserts that there exist finitely many QA links of a given determinant.  If true, then it would immediately imply our main results, Theorem \ref{thm:main} and Proposition \ref{prp:qa}.

\bibliographystyle{plain}
\bibliography{bibliography}

\end{document}